\documentclass{amsart}
\usepackage[utf8]{inputenc}
\usepackage{amssymb,latexsym}
\usepackage{amsmath}
\usepackage{graphicx}
\usepackage{hyperref}
\usepackage{textcomp}
\usepackage[dvipsnames]{xcolor}
\usepackage{tabularx}
\usepackage{array}

\usepackage{amsthm,amssymb,enumerate,graphicx, tikz}
\usepackage{amscd}
\usepackage{setspace}
\usepackage{comment}
\usepackage{hyperref}
\usepackage{cleveref}
\usepackage{caption}
\usepackage{eucal}

\def\@logofont{\footnotesize}
\def\@setaddresses{\par
  \nobreak \begingroup
  \footnotesize
  \def\author##1{\nobreak\addvspace\bigskipamount}%
  \def\\{\par\nobreak}%
  \interlinepenalty\@M
  \def\address##1##2{\begingroup
    \par\addvspace\bigskipamount\indent
    \@ifnotempty{##1}{(\ignorespaces##1\unskip) }%
    {\scshape\ignorespaces##2}\par\endgroup}%
  \def\curraddr##1##2{\begingroup
    \@ifnotempty{##2}{\nobreak\indent\curraddrname
      \@ifnotempty{##1}{, \ignorespaces##1\unskip}\/:\space
      ##2\par}\endgroup}%
  \def\email##1##2{\begingroup
    \@ifnotempty{##2}{\nobreak\indent\emailaddrname
      \@ifnotempty{##1}{, \ignorespaces##1\unskip}\/:\space
      \ttfamily##2\par}\endgroup}%
  \def\urladdr##1##2{\begingroup
    \def~{\char`\~}%
    \@ifnotempty{##2}{\nobreak\indent\urladdrname
      \@ifnotempty{##1}{, \ignorespaces##1\unskip}\/:\space
      \ttfamily##2\par}\endgroup}%
  \addresses
  \endgroup
}
\renewcommand*\subjclass[2][2010]{%
  \def\@subjclass{#2}%
  \@ifundefined{subjclassname@#1}{%
    \ClassWarning{\@classname}{Unknown edition (#1) of Mathematics
      Subject Classification; using '2000'.}%
  }{%
    \@xp\let\@xp\subjclassname\csname subjclassname@#1\endcsname
  }%
}

\usepackage{graphicx,amsmath,amssymb}
\usepackage{algorithm}
\usepackage{mathdots, comment}
\usepackage[noend]{algpseudocode}

\newtheorem{theorem}{Theorem}[section]

\newtheorem*{theorem*}{Theorem}
\newtheorem{proposition}[theorem]{Proposition}
\newtheorem{lemma}[theorem]{Lemma}

\newtheorem{corollary}[theorem]{Corollary}

\theoremstyle{definition}
\newtheorem{definition}[theorem]{Definition}

\theoremstyle{remark}
\newtheorem{remark}[theorem]{Remark}

\begin{document}
\title[On Partially Matchable Subspaces]{On Partially Matchable Subspaces in a Field Extension}

\author[]{Mohsen Aliabadi$^{1}$ \and Jozsef Losonczy$^{2,*}$}
\thanks{$^1$Department of Mathematics, Clayton State University, 
2000 Clayton State Boulevard, Lake City, Georgia 30260, USA.  \url{maliabadi@clayton.edu}.\\
$^2$Department of Mathematics, Long Island University,
720 Northern Blvd, Brookville, New York 11548, USA. \url{Jozsef.Losonczy@liu.edu}.}
\thanks{$^*$Corresponding Author.}


\begin{abstract}
We develop a theory of partial matchings between finite-dimensional subspaces $A,B$ in a field extension $K \subsetneq L$, providing structural, quantitative, and extremal results for their deficiency.  Our main results include (1) a characterization of those pairs $(A,B)$ that are partially matchable up to a specified defect, (2) a decomposition theorem for pairs $(A,B)$ having positive deficiency, (3) an existence criterion for pairs having a prescribed dimension and satisfying a deficiency bound, and (4) a formula for the maximum attainable deficiency at each dimension, assuming $1 \notin B$ and the extension has a proper nontrivial intermediate field of finite $K$-dimension. We use these results to recover and extend various parts of this area of matching theory. Our approach blends algebraic techniques with tools from matroidal transversal theory, and utilizes a linearized version of the $e$-transform from additive number theory.
\end{abstract}

\maketitle

\section{Introduction}\label{Intro}

Let \(G\) be an abelian group, with operation written multiplicatively.  For nonempty finite sets \(A,B\subseteq G\), a \emph{matching} is a bijection \(f:A\to B\) such that \(af(a)\notin A\) for all \(a\in A\). If such a mapping exists, we say that the pair \( (A,B) \) is \emph{matchable}; otherwise, it is called \emph{unmatchable}.  Matchings in groups arose in \cite{Losonczy 1} in connection with a conjecture of E. K. Wakeford on canonical forms for homogeneous polynomials \cite{Wakeford}.  The conjecture concerns the following question: Which sets of monomials are removable from a generic homogeneous polynomial through a linear change in its variables?  General frameworks for approaching this and related problems can be found in \cite{Ehrenborg Rota} and \cite{Ehrenborg}. In \cite{Losonczy 1}, progress on the Wakeford problem was made by establishing the existence of a specific type of matching, called an acyclic matching, for certain pairs of subsets of the additive group \(\mathbb{Z}^n\). 

More concretely, let $A$ be a set of monomials of fixed total degree in the variables \( x_1,\dots,x_n \), with \( |A|=n(n-1) \).  One associates with \( A \) a weighted bipartite graph whose biadjacency matrix \( \mathcal{M}_A \) records how the monomials in \( A \) interact with the differential operators \( x_i\frac{\partial}{\partial x_j} \) for  \( i\neq j \). It was proved in~\cite{Losonczy 1} that \( A \) is generically removable if and only if the matrix \( \mathcal{M}_A \) is nonsingular. This criterion admits a transparent combinatorial interpretation: In the Leibniz expansion of the determinant of \( \mathcal{M}_A \), each nonzero term corresponds to a matching $A\to B$, where the monomials in $A$ are identified with their exponent vectors, and \( B \) is a set of vectors in natural bijection with \( \{x_i\frac{\partial}{\partial x_j} : i\neq j\}\).  It turns out that when every monomial in $A$ involves all $n$ variables, there is an ``acyclic matching" whose associated Leibniz term cannot be canceled by other terms.  Thus such a matching provides a combinatorial certificate for the removability of the monomials in \( A \).

In subsequent work, it was gradually clarified when such matchings occur: The acyclic matching property was established for $\mathbb{Z}^n$ in \cite{Alon}, then extended to all torsion-free abelian groups in \cite{Losonczy 2} by exploiting a total ordering compatible with the group structure \cite{Levi}. The property was characterized for sufficiently small subsets of an arbitrary abelian group in \cite{Aliabadi 0} and then ultimately settled for the full class of abelian groups in \cite{Taylor}.

Since the 1990s, the theory has expanded in several directions, including sharper results for abelian groups \cite{Aliabadi 0} and a generalization to arbitrary (not necessarily abelian) groups \cite{Eliahou 1}. The existence of matchings can be viewed as a consequence of product-set expansion phenomena, as described by extensions of the Scherk--Kemperman theorem obtained by Hamidoune via the isoperimetric method \cite{Hamidoune 3}. Further developments include quantitative and enumerative aspects \cite{Hamidoune}, a linear formulation over field extensions \cite{Eliahou 2}, and a matroidal analogue \cite{Aliabadi 3}. 

Let \(G\) be an abelian group and let \(A,B\subseteq G\) be nonempty finite subsets with \(|A|=|B|\). A \emph{partial matching} of the pair \( (A,B) \) with \emph{defect} \(d\) is defined in  \cite{Aliabadi 7} to be an injective mapping \(f:A'\to B\), where \(A'\subseteq A\) and \(d=|A\setminus A'|\), such that
\[
af(a)\notin A \qquad \text{for all } a\in A'.
\]
When such a mapping exists, one also says that \(A\) is \emph{partially matched} to \(B\) with \emph{defect} \(d\). 

Below are three of the main results on partially matched sets in abelian groups from \cite{Aliabadi 7}. The first theorem provides a necessary and sufficient condition for the existence of a partial matching with a prescribed defect.

\begin{theorem}  \label{Char for partial in groups}
Let \( A \) and \( B \) be nonempty finite subsets of an abelian group \( G \), with \( |A| = |B| \) and \( 1 \notin B \). Let \( 0 \leq d \leq |A| \) be an integer. Then there exists a partial matching of \( (A,B) \) with defect \( d \) if and only if, for every pair of nonempty subsets \( S \subseteq A \) and \( R \subseteq B \cup \{ 1 \} \) with \( SR = S \), we have 
\[ |S| - d \leq |B \setminus R|. 
\]
\end{theorem}

In the following theorem, \( \delta(A,B) \) stands for the smallest nonnegative integer \( d \) such that \( (A,B) \) has a partial matching with defect \( d \).  One calls \( \delta(A,B) \) the {\it deficiency} of the pair \( (A,B) \). This theorem provides a structural description of unmatchable pairs.

\begin{theorem} \label{Structure of PMG}
Let \( A \) and \( B \) be nonempty finite subsets of an abelian group \( G \), with \( |A| = |B| \) and \( 1 \notin B\).  Let \(\ell \) be a nonnegative integer.  Then \( \delta(A,B) > \ell \) if and only if there exists a nonempty subset \( R \) of \( B \) such that \( A \) and \( B \) can be expressed as disjoint unions as follows:
\begin{align*}
A &= S \cup Y, \\ 
B &= R \cup Z,
\end{align*}
where \( S \) is a disjoint union of cosets of \( \langle R \rangle \), and \( Y \)satisfies \( |Y| < | R | - \ell\). 
\end{theorem}

For an abelian group \(G\) with at least one finite nontrivial proper subgroup, denote by \(n_0(G)\) the minimum order of such a subgroup. The theorem below gives a necessary and sufficient condition for the existence of subsets \(A,B\subseteq G\) of a prescribed size and with deficiency \(\delta(A,B)\) exceeding a given value \(\ell\), while also requiring \( 1 \notin B \).

\begin{theorem} \label{Existence partial}
Let $G$ be an abelian group with at least one finite nontrivial proper subgroup. Let \( \ell \) and \( n \) be integers such that \( \ell \ge 0 \) and \( n_0(G) \le n<|G| \). Then there exist subsets $A,B\subseteq G$ such that \( 1\notin B \), \( |A|=|B|=n \), and \( \delta(A,B) > \ell \) if and only if there is a subgroup \( H \) of \( G \) with \( |H|\le n \) and \( |H|\nmid (n+j) \) for all \( j = 1, \ldots, \ell+1 \).
\end{theorem}

Theorem~\ref{Char for partial in groups} was obtained using Dyson’s $e$-transform, which serves as a fundamental reduction device in the combinatorial analysis of product sets. More precisely, given finite nonempty subsets \(A\), \(S\), and \(R\) of an abelian group \(G\) with \(1\in R\) and \(SR\subseteq A\), there exist subsets \(S'\) and \(R'\) such that
\begin{itemize}
\item[(i)] \(S \subseteq S'R' = S' \subseteq A\),
\item[(ii)] \(1 \in R' \subseteq R\),
\item[(iii)] \(|S'| + |R'| = |S| + |R|\).
\end{itemize}
A proof of this well-known fact can be found in \cite{Aliabadi 7}.  The transformation \( (S,R) \mapsto (S',R') \) preserves the additive quantity \(|S|+|R|\) while replacing \((S,R)\) with a more structured pair \((S',R')\), and it is precisely this feature that plays an important role in the proof of the first theorem above. 

We have three main objectives in this paper. First, we formulate a notion of partial matchability for finite-dimensional subspaces in a field extension. Second, we deploy a linearization of Dyson’s $e$-transform to establish linear counterparts to Theorems \ref{Char for partial in groups}, \ref{Structure of PMG}, and \ref{Existence partial}.  Finally, we derive several structural and extremal consequences of these theorems.

The paper is organized as follows. Section~\ref{prelim} introduces the general framework of partial matchability in a linear setting and gives the main definitions and tools used throughout. In Section~\ref{PartialFieldSec}, we recall a linear analogue of Dyson’s $e$-transform (Proposition~\ref{Linear e transform}) and use it to characterize partial matchability for pairs of subspaces (Theorem \ref{LinPartial}). We also give an explicit deficiency formula for partially matchable $n$-dimensional subspaces \( A \) and \( B \), answering a conjecture posed in \cite{Aliabadi 4} (Theorem \ref{DefTh} and Remark~\ref{ConjRem}).  In Section~\ref{LinStructure}, we first prove a structural decomposition theorem for partially matchable pairs with positive deficiency (Theorem~\ref{LinStrucTh}). Then we establish a criterion guaranteeing the existence of a pair of subspaces of prescribed dimension in a field extension, where the deficiency of the pair is required to be greater than a specified threshold \( \ell \) (Theorem~\ref{ExistancePartial}). We also derive an explicit formula for the maximal deficiency attainable at a prescribed dimension, expressed in terms of the intermediate degrees of the ambient field extension (Theorem~\ref{MaxDef}).

\section{Preliminaries}\label{prelim}

We begin by establishing some basic terminology and notation. For any positive integer \( n \), we write \( [n] \) for the set \( \{1, \ldots, n\} \).  
Given a subset \( S \) of a vector space \( V \), we let \( \langle S \rangle \) denote the subspace of \( V \) spanned by \( S \).  If \( S = \{x_1, \ldots, x_n\} \), we also write \( \langle x_1, \ldots, x_n \rangle \) for this subspace. 

Suppose \( K \subseteq L \) is a field extension and \( F \) is an intermediate subfield.  If \( U \) is an $F$-subspace of \( L \), we write \( \dim_F U \) for its dimension.  In situations where \( F \) is known to be \( K \), we usually simplify this by writing \( \dim U \).

If \( A,B \) are \( K \)-subspaces of \( L \), we use \( AB \) to denote the \emph{Minkowski product} of \( A \) and \( B \):
\[
AB = \{ ab : a \in A,\ b \in B \}.
\]

In \cite{Eliahou 2}, a notion of matching for subspaces in a skew field extension was introduced and investigated.  Below, we formulate a generalization of the definition in \cite{Eliahou 2}, although we restrict our attention to commutative extensions.  The idea here is to consider situations where a pair of subspaces may not be matched, but are in a precise sense partially matched. The definition proceeds in two steps.

Let \( K \subsetneq L \) be a field extension, and let \( A \) and \( B \) be two \( n \)-dimensional $K$-subspaces of \( L \), with \( n \) a positive integer. Let \( 0 \leq d \leq n \) be an integer. A basis \( \mathcal{A} = \{a_1, \dots, a_n\} \) of \( A \) is said to be {\it{partially matched}} to a basis \( \mathcal{B} = \{b_1, \dots, b_n\} \) of \( B \) with {\it defect} \(d\) if there exists a subset \( P \) of \( [n] \) such that \( |P| = n-d \) and
\[
a_i^{-1} A \cap B \subseteq \langle b_1, \dots, b_{i-1}, b_{i+1}, \dots, b_n \rangle \quad \text{for each } i \in P.
\]

\begin{remark}
Suppose that the above condition holds, and let \(\mathcal{A}' = \{ a_i : i \in P \}\). Consider the injection \(f : \mathcal{A}' \to \mathcal{B}\) given by \(f(a_i) = b_i\).  Then \(a_i f(a_i) \notin \mathcal{A}\) for all \(i \in P\), so that, in the group-theoretic sense, \( f \) is a partial matching of \( (\mathcal{A}, \mathcal{B})\) with defect \( d \), if we view \(\mathcal{A} \) and \(\mathcal{B} \) as subsets of the multiplicative group \(L^{\times}\).
\end{remark}

\begin{remark}  \label{Defect Property}
If \( \mathcal{A} \) is partially matched to \( \mathcal{B} \) with defect \( d \), then for any integer \( d' \) satisfying \( d \leq d' \leq n \), it is easy to see that \( \mathcal{A} \) is partially matched to \( \mathcal{B} \) with defect \( d' \).
\end{remark}

Next, we say that \( A \) is {\it{partially matched}} to \( B \) with {\it{defect \( d \)}} if for every ordered basis \( \mathcal{A} \) of \( A \), there exists an ordered basis \( \mathcal{B} \) of \( B \) such that \( \mathcal{A} \) is partially matched to \( \mathcal{B} \) with defect \( d \). If \(A\) is partially matched to \(B\) with defect \(d = 0\), we will also say, more simply, that \(A\) is {\em matched} to \( B \), or that the pair \( (A,B) \) is {\em matchable}.  

One of the main results of this paper will be a characterization of those pairs \( (A,B) \) such that \( A \) is partially matched to \( B \) with a given defect (Theorem~\ref{LinPartial}).  To prove it, we will find it helpful to work with free partial transversals. We explain this concept next.

Let \( V \) be a vector space over a field \( K \), and let \( \mathcal{W} = ( W_1, \ldots , W_n ) \) be a family of $K$-subspaces of \( V \) (the \( W_i \) are not assumed to be distinct). Let \( 0 \leq d \leq n \) be an integer. A \emph{free partial transversal} of \( \mathcal{W} \) with \emph{defect} \( d \) is a linearly independent subset \( X \) of \( V \) of size \(  n-d \) such that there is an injective mapping \( \theta : X \rightarrow [n] \) satisfying \( x \in W_{\theta(x)} \) for all \( x \in X \).

We will need a method for determining whether a family \( \mathcal{W} \) actually has a free partial transversal with specified defect.  A useful one is provided by the following well-known consequence of the defect form of Rado's theorem (see Theorem 6.2.2 in \cite{Mirsky} and Remark~\ref{defect form Rado} below).

\begin{theorem}\label{Linear Partial}
Let \( V \) be a finite-dimensional vector space over a field \( K \), and consider a family \( \mathcal{W} = ( W_1, \ldots , W_n ) \) of subspaces of \( V \).  Let \( 0 \leq d \leq n \) be an integer. Then \( \mathcal{W} \) admits a free partial transversal with defect \( d \) if and only if
\[
\dim \sum_{i \in J} W_i \geq |J| - d \quad \text{for all } J \subseteq [n].
\]
\end{theorem}

\begin{remark}  \label{defect form Rado}
In \cite{Mirsky}, Rado's theorem and its defect version are formulated in the more general setting of a pre-independence structure \( \mathcal{E} \) on a nonempty set \( E \) with associated rank function \( \rho \).  To recover Theorem~\ref{Linear Partial} above, choose a basis \( A_i \) for each subspace \( W_i \), set \( E = \bigcup_i A_i \), and let \( \mathcal{E} \) be the collection of all linearly independent subsets of \( E \). It is then easy to verify that \( \mathcal{E} \) is a pre-independence structure on \( E \) with rank function \( \rho \) satisfying \( \rho(S) = \mbox{dim} \langle S \rangle \) for each \( S \subseteq E \).  One then applies Theorem~6.2.2 in \cite{Mirsky} to the family \( (A_1, \ldots ,A_n) \).
\end{remark}

\section{Partial matchings in field extensions}\label{PartialFieldSec}

We begin by describing a linear analogue of Dyson's $e$-transform from additive number theory. This linear version was used in the proof of  Theorem~3.1 in \cite{Aliabadi 5}. We are not aware of its appearance elsewhere in the literature.  We present it here in isolation for the convenience of the reader, and to shorten the proof of the main result of this section, Theorem~\ref{LinPartial}.

\begin{proposition} \label{Linear e transform}
Let \( K \subseteq L \) be a field extension, and let \( A \), \( S \), and \( R \) be nonzero $K$-subspaces of \( L \) of finite dimension.  Assume that \( 1 \in R \) and \( SR \subseteq A \).  Then there exist nonzero $K$-subspaces \( S' \) and \( R' \) of \( L \) such that the following conditions hold:
\begin{itemize}
\item[(i)]  \( S'R' = S' \subseteq A \),
\item[(ii)]  \( 1 \in R' \subseteq R \),
\item[(iii)] \( \dim S' + \dim R' =  \dim S + \dim R \).
\end{itemize}
\end{proposition}

\begin{proof}
If \( SR = S \), we can simply take \( S' = S \) and \( R' = R \), so assume \( SR \neq S \).  Then \( SR \not\subseteq S \). Choose \( e \in S \) and \( r \in R \) such that \( er \in SR\setminus S \). Define subspaces \( S_1 \) and \( R_1 \) as follows:
\begin{align*}
S_1 &= S + eR, \\
R_1 &= R \cap (Se^{-1}).
\end{align*}
We claim that the following conditions hold:
 
\begin{enumerate}
    \item[(a)] \( S_1 R_1 \subseteq \langle SR\rangle \subseteq A \),
    \item[(b)] \( 1 \in R_1 \subseteq R \),
    \item[(c)] \( \dim S_1 + \dim R_1 = \dim S + \dim R \),
    \item[(d)] \( S_1 \subseteq A \) and \( \dim S < \dim S_1 \).
\end{enumerate}
Conditions (a) and (b) follow directly from the definitions of \( S \), \( R \), \( S_1 \), and \( R_1 \). 

To verify (c), we first apply Grassmann's dimension formula for vector subspaces:
\begin{align*}
\dim S_1 &= \dim(S + eR) \\
         &= \dim S + \dim eR - \dim(S \cap eR) \\
         &= \dim S + \dim R - \dim(S \cap eR).
\end{align*}
Now observe that the mapping \( x \mapsto ex \) defines a linear isomorphism from \( R \cap (S e^{-1}) \) to \( S \cap eR \), since \( e \neq 0 \). This implies
\[
\dim R_1 = \dim\big(R \cap (S e^{-1})\big) = \dim(S \cap eR),
\]
so
\[
\dim S_1 + \dim R_1 = \dim S + \dim R,
\]
confirming (c).

Concerning (d), the inclusion \( S_1 \subseteq A \) follows from (a) and the fact that \( 1 \in R_1 \).  For the dimension inequality, note that \( S_1 \) properly contains \( S \), since \( e r \in S_1 \setminus S \). 

All four conditions have been verified. Now if \( S_1 R_1  \neq S_1 \), we repeat the above, substituting \( S_1 \) for \( S \) and \( R_1 \) for \( R \). The process continues until we obtain nonzero subspaces \( S_m \) and \( R_m \) such that \( S_m R_m = S_m \). Termination is guaranteed, since \( \dim A < \infty \) and the sequence of subspaces \( S_1, S_2, \ldots \) strictly increases in dimension. Let \( S' = S_m \) and \( R' = R_m \).  These subspaces clearly satisfy the conditions in the statement.
\end{proof}

The following lemma will be used in the proof of our first theorem on partial matchability.

\begin{lemma}\label{setsize}
Let $n$, $\ell$, and $d$ be nonnegative integers with $\max\{\ell,d, 1 \}\leq n$. Let $P$ be a subset of $[n]$ of size \( n-d \). Then $n-|[\ell]\cap P|\leq n-\ell+d$.
\end{lemma}

\begin{proof}
The desired inequality is equivalent to \( |[\ell] \cap P| \ \ge \ \ell - d \), which obviously holds if \( d \geq \ell \).  Suppose \( d < \ell \), and assume for a contradiction that \( |[\ell] \cap P| < \ \ell - d \). Then
\[
|P| = |[\ell] \cap P| + |([n]\setminus [\ell]) \cap P| < \ell -d + n - \ell = n - d,
\]
contrary to the assumption that \( |P| = n-d \).  The lemma is proved.    
\end{proof}

To prove the main theorem of this section, we need to fix some standard notation. Let \(V\) be a $K$-vector space, and write \( V^* \) for its dual space, i.e., the vector space consisting of all $K$-linear mappings \( f:V\to K \).  For a subspace \(W\subseteq V\), its annihilator \(W^\perp\subseteq V^*\) is defined by
\[
W^\perp=\{f\in V^* : W\subseteq \ker f \}.
\]
If \(V\) is finite-dimensional, then we have the following two well-known identities for annihilators: given subspaces \(U_1, \ldots ,U_m,W\) of \( V \),
\[
\Bigl(\bigcap_{i=1}^m U_i\Bigr)^{\perp} = \sum_{i=1}^m U_i^\perp 
\]
and
\[
\dim W^\perp=\dim V-\dim W.
\]
In what follows, all direct sums will be internal.

With the above preparations, we now present the main result of this section.

\begin{theorem}\label{LinPartial}
Let \( K \subsetneq L \) be a field extension, and let \( A \) and \( B \) be two \( n \)-dimensional \( K \)-subspaces of \( L \), with \( 0 < n < \infty \) and \( 1 \notin B \). Let $0 \leq d \leq n$ be an integer. Then $A$ is partially matched to $B$ with defect $d$ if and only if, for every pair of nonzero $K$-subspaces \( S\subseteq A \) and \( R\subseteq B \oplus K \) with \( SR = S\), we have 
\[
\dim S-d\leq \dim \big(B/(B\cap R)\big).
\]
\end{theorem}

\begin{proof}
Assume that \(A\) is partially matched to \(B\) with defect \( d \). Let \(S \subseteq A \) and \(R \subseteq B\oplus K \) be nonzero $K$-subspaces with \( SR  = S \). Choose a basis \( \{ a_1, \ldots , a_{\ell} \} \) for \( S \), and extend this to a basis \( \mathcal{A}=\{a_1,\dots,a_\ell,a_{\ell+1},\dots,a_n\} \) for \( A \).  Since \(A\) is partially matched to \(B\) with defect \(d\), there exists a basis \(\mathcal{B}=\{b_1,\dots,b_n\}\) of \(B\) such that \(\mathcal{A}\) is partially matched to \(\mathcal{B}\) with defect \(d\). Thus there exists a subset \(P \) of \( [n] \) of size \( n-d \) such that 
\[
a_i^{-1}A\cap B \;\subseteq\; \langle b_1,\dots,b_{i-1},b_{i+1},\dots,b_n\rangle
\quad\text{for each } i\in P.
\]

Now, we have
\begin{align}\label{eq1}
\bigcap_{i\in[\ell]} \big(a_i^{-1}A \cap B\big)
&\subseteq \bigcap_{i\in[\ell] \cap P} \big(a_i^{-1}A \cap B\big) \\
&\subseteq \bigcap_{i\in[\ell] \cap P} \langle b_1, \dots, b_{i-1}, b_{i+1}, \dots, b_n\rangle, \nonumber
\end{align}
where the last two intersections are understood to equal \( B \) if \( [\ell] \cap P \) is empty. Clearly
\[
\dim \! \left( \bigcap_{i\in[\ell] \cap P} \langle b_1, \dots, b_{i-1}, b_{i+1}, \dots, b_n\rangle \right)
= n - |[\ell] \cap P|.
\]
Combining this with (\ref{eq1}), we get
\[
\dim \bigcap_{i\in[{\ell}]}(a_i^{-1}A\cap B) \leq n-|[\ell]\cap P|.
\]

On the other hand, since \( SR = S \), we have \( a_i R \subseteq S \) for each \( i \in [\ell] \), and hence  
\[
R \subseteq a_i^{-1} S \subseteq a_i^{-1} A.
\]
Therefore
\[
R \cap B \;\subseteq\; \bigcap_{i\in[\ell]} \big(a_i^{-1} A \cap B\big).
\]
Thus, according to Lemma \ref{setsize},  
\[
\dim (R\cap B)\leq \dim \bigcap_{i\in[{\ell}]}(a_i^{-1}A\cap B) \leq n-\ell+d,
\]
which implies the desired inequality
\[
\dim S-d\leq \dim \big(B/(B\cap R)\big).
\]

Conversely, assume that the condition in the statement involving \( S \) and \( R \) holds.  
We will show that \( A \) is partially matched to \( B \) with defect \( d \).  
Let \(\mathcal{A} = \{a_1,\dots,a_n\}\) be a basis for \(A\), and let \(J \subseteq [n]\) be nonempty.  
Set
\[
S = \langle a_i : i \in J \rangle, \quad
T = \bigcap_{i \in J} (a_i^{-1} A \cap B), \quad
R = T + K,
\]
and note that the sum \( T + K \) is direct.  We claim that \(\dim T \leq n - |J| + d\).  To verify this, we consider two cases.

\medskip

\textbf{Case 1:} \( SR = S\).  
By hypothesis,
\[
\dim S - d \;\leq\; \dim \big(B/(B\cap R)\big).
\]
Since \(T \subseteq B\), \(R = T \oplus K\), and \(1 \notin B\), it follows that
\(T = B \cap R\). Therefore
\[
\dim T \;=\; \dim(B \cap R) \;\leq\; \dim B - \dim S + d \;=\; n - |J| + d,
\]
as claimed.

\medskip

\textbf{Case 2:} \( SR \neq S \). We wish to apply Proposition~\ref{Linear e transform}. By construction, we have $ST\subseteq A$. Hence, since $R=T\oplus K$, 
\[
SR=S(T\oplus K) \subseteq ST + SK \subseteq \langle ST\cup S\rangle \subseteq A.
\]
Therefore, by Proposition~\ref{Linear e transform}, there exist nonzero $K$-subspaces \( S' \) and \( R' \) of \( L \) such that 
\begin{itemize}
\item[(a)]  \( S'R' = S' \subseteq A \),
\item[(b)]  \( 1 \in R' \subseteq R \subseteq B \oplus K \),
\item[(c)] \( \dim S' + \dim R' =  \dim S + \dim R \).
\end{itemize}
By applying our hypothesis to \( S' \) and \( R' \), which we are able to do on account of (a) and (b), we obtain
\[
\dim S' - d \;\leq\; \dim \big(B/(B\cap R')\big).
\]
Now, it is easy to see that \( R' = (B \cap R')\oplus K\). Hence \( \dim (B \cap R') = \dim R' - 1 \), and so we have the inequality  
\[
\dim S'- d \;\leq\; n - \dim R' + 1,
\]
which, combined with (c) and the fact that \( \dim R = \dim T + 1 \), yields
\[
\dim T \;\leq\; n - |J|+d,
\]
as claimed.

So in both cases we have \( \dim T \leq n - |J|+d \). Passing to the annihilator in the dual space \( B^* \), we obtain, by way of one of the annihilator identities,
\[
\dim T^\perp \geq |J|-d,
\]
which leads, using the other annihilator identity, to
\[
\dim \sum_{i \in J} (a_i^{-1} A \cap B)^\perp  \geq |J| - d. 
\]
By Theorem~\ref{Linear Partial}, the family
\[
\mathcal{W} = \big( (a_1^{-1} A \cap B)^\perp, \dots, (a_n^{-1} A \cap B)^\perp \big)
\]
admits a free partial transversal with defect \(d\).  Thus there is a subset \( P \)  of \( [n] \) and a linearly independent subset   
\( \{ f_i : i \in P \} \) of \( B^* \), both of size \( n - d \), such that  
\begin{align}\label{eq2}
    f_i \in (a_i^{-1} A \cap B)^\perp \quad \text{for each } i \in P. 
\end{align}

Extend \( \{ f_i : i \in P \} \) to a basis \(\mathcal{F} = \{ f_1, \dots, f_n \}\) of \(B^*\), and let \(\mathcal{B} = \{ b_1, \dots, b_n \}\) be the basis of \(B\) dual to \(\mathcal{F}\). We show that \(\mathcal{A}\) is partially matched to \(\mathcal{B}\) with defect \(d\). Since \(f_i(b_j) = \delta_{ij}\), it follows that
\[
\ker f_i = \langle b_1, \dots, b_{i-1}, b_{i+1}, \dots, b_n \rangle
\quad \text{for each } i \in P.
\]
Combining this with \eqref{eq2} yields
\[
a_i^{-1} A \cap B
\subseteq \langle b_1, \dots, b_{i-1}, b_{i+1}, \dots, b_n \rangle
\quad \text{for each } i \in P.
\]
Hence \(\mathcal{A}\) is partially matched to \( \mathcal{B} \) with defect \( d \). We have thus shown that \( A \) is partially matched to \( B \) with defect \( d \), completing the proof.
\end{proof}

In the remainder of this section, we explore several consequences of Theorem~\ref{LinPartial}. To state our results, we will need some standard notation and terminology from field theory. Let \(K\subseteq L\) be a field extension and let \(R \) be a subset of \( L \). We write \( K(R) \) for the smallest subfield of \(L\) containing \(K\cup R\). If \(R=\{x\}\) (for some \( x \in L\)), we write \(K(x)\) instead of \(K(\{x\})\). We say that \(x\) is \emph{algebraic} over \(K\) if \(K(x)\) has finite dimension as a $K$-vector space. Given an intermediate field \( K \subseteq F \subseteq L \), we denote the $K$-dimension of \( F \), called the \emph{degree} of \(F\) over \( K \), by \( [F:K] \). If \( F = K(x) \), we also refer to this value as the degree of \( x \) over \( K \).

The following proposition explains how the condition \( SR = S \) will be used in our applications of Theorem~\ref{LinPartial}. A proof can be found in \cite{Aliabadi 6}.

\begin{proposition} \label{proposition linear}
Let \( K \subseteq L \) be a field extension, let \( n \) be a positive integer, and let \( S \) and \( R \) be nonzero \( K \)-subspaces of \( L \).  Assume that \( SR = S \) and  \( \dim_K S \leq n \). The following statements hold: 
\begin{itemize}
\item[(i)] Let \( x \in R \). Then \( aK(x) \subseteq S \) for all \( a \in S \). In addition, \( [K(x):K] \leq n\), so that \( x \) is algebraic over $K$.    
\item[(ii)] We in fact have \( aK(R) \subseteq S \) for all \( a \in S \). Also, \( [K(R):K] \leq n \) and \( \dim_K S \) is a positive multiple of \( [K(R):K] \).
\end{itemize}
\end{proposition}

Let \( B \) be an $n$-dimensional $K$-subspace in a field extension \( K \subsetneq L \), with \( n > 0 \) and \( 1 \notin B \). For our first application of Theorem~\ref{LinPartial}, we show that if \( B \) does not contain \( d + 1 \) linearly independent elements of degree \( \leq n \) over \( K \), then any $n$-dimensional subspace of \( L \) can be partially matched to it with defect \( d \).

\begin{corollary}
Let $K \subsetneq L$ be a field extension and let $A,B \subseteq L$ be $n$-dimensional $K$-subspaces, with \( 0 < n < \infty \) and  \( 1 \notin B \).  Let \( 0 \leq d \leq n \) be an integer. Suppose that 
\[
\dim \langle\, b \in B : [K(b):K] \leq n \,\rangle \leq d.
\]
Then $A$ is partially matched to $B$ with defect $d$.
\end{corollary}

\begin{proof}
Define
\[
D=\langle\, b \in B : [K(b):K] \le n \,\rangle.
\]
Let \( S \subseteq A \) and \( R \subseteq B \oplus K \) be nonzero $K$-subspaces with \( SR = S \). We claim that \( B \cap R \subseteq D \). Indeed, if \( x \in L \setminus D \), then \( x \notin B \) or \( [K(x):K] > n \), the latter implying \( x \notin R \) by Proposition~\ref{proposition linear}.  

Therefore \( \dim (B \cap R) \leq d \), and so 
\[
\dim S - d \;\le\; n-d \;\le\; \dim B - \dim(B\cap R)
\;=\; \dim \big( B / (B\cap R) \big).
\]
By Theorem~\ref{LinPartial}, $A$ is partially matched to $B$ with defect $d$.
\end{proof}

\begin{definition}  \label{delta definition}
Let \(K\subsetneq L\) be a field extension and let \(A,B\subseteq L\) be finite-dimensional $K$-subspaces with \(\dim A=\dim B > 0\) and \(1\notin B\). We define \(\delta_{L/K}(A,B)\) to be the least nonnegative integer \(d\) for which \(A\) is partially matched to \(B\) with defect \(d\). When the ambient extension is clear from the context, we write \(\delta(A,B)\) in place of \(\delta_{L/K}(A,B)\).
\end{definition}

\begin{corollary}
Let \(K \subsetneq L\) be a field extension, and let \(A\) and \(B\) be two \(n\)-dimensional \(K\)-subspaces of \(L\), with \( 0 < n < \infty \) and \(1 \notin B\).  Then \( \delta(A,B) \) equals the maximum value of 
\[
\dim S - \dim \big(B/(B\cap R)\big), 
\]
where \( S \) and \( R \) range over all nonzero subspaces of \( A \) and \( B \oplus K \), respectively, such that \( SR = S \).
\end{corollary}

\noindent {\it Note: }The maximum value in the statement is at least \( 0 \), since we can take \( S = A \) and \( R = K \), and is clearly no greater than \( n \).

\begin{proof}
Let \( d = \delta(A,B) \) and define \( \mathcal{I} \) to be the set of all differences 
\[ \dim S - \dim \big(B/(B\cap R)\big), 
\] 
where \( S \) and \( R \) are as in the statement.

Since \( A \) is partially matched to \( B \) with defect \( d \), it follows from Theorem~\ref{LinPartial} that for all such \( S \) and \( R \),
\[
\dim S - d \;\le\; \dim \big(B/(B\cap R)\big).
\]
Hence \( d \ge \max \mathcal{I} \).

For the reverse inequality, let \(d' = \max \mathcal{I}\). Then for all nonzero subspaces \(S \subseteq A\) and \(R \subseteq B \oplus K\) with \( SR = S\), we have
\[
\dim S - d' \;\le\; \dim \big(B/(B\cap R)\big).
\]
Applying Theorem~\ref{LinPartial} again, \(A \) is partially matched to \( B \) with defect \( d' \). By the minimality of \(d = \delta(A,B)\), it follows that \(d \le d'\).  
Combining this with the previous paragraph gives \(d = d'\), as claimed.
\end{proof}

The next result relates the notion of defect to the dimension of the intersection of the two given subspaces \( A \) and \( B \).

\begin{corollary}
Let \( K \subsetneq L \) be a field extension, and let \( A \) and \( B \) be two \(n\)-dimensional \(K\)-subspaces of \(L\), with \( 0 < n < \infty \) and \((A + B) \cap K = \{0\}\).  Let \(d = \dim A - \dim (A \cap B)\). Then \(A\) is partially matched to \(B\) with defect \(d\).
\end{corollary}

\begin{proof}
Let \(S\subseteq A\) and \(R\subseteq B\oplus K\) be nonzero $K$-subspaces with \( SR = S\). We claim that
\[
\dim S - d \;\leq\; \dim \big(B/(B\cap R)\big).
\]
First, note that \(R\cap S = \{0\}\). Indeed, if there existed a nonzero \(a \in R\cap S\), then by Proposition~\ref{proposition linear} we would have
\(aK(a) \subseteq S\), implying \( 1 \in S \subseteq A \), contrary to our assumption that \((A + B) \cap K = \{0\}\).

Using the above, together with Grassmann's dimension formula and the fact that 
\[ 
S + (B\cap R) \subseteq A + B,
\] 
we get 
\begin{align*}
\dim S + \dim (B \cap R) = \dim \big( S + (B \cap R) \big) &\leq \dim (A+B) \\ 
&= \dim A + \dim B - \dim (A\cap B) \\
&= \dim B + d.
\end{align*}
Therefore
\[
\dim S - d
\;\leq\;
\dim B - \dim (B\cap R) = \dim \big(B/(B\cap R)\big).
\]
By Theorem~\ref{LinPartial}, \( A \) is partially matched to \( B \) with defect \( d \).
\end{proof}

Below we define Chowla and $d$-defective Chowla subspaces.  We will then explain how they fit into the theory of partial matchability.

\begin{definition}
Let \( K \subsetneq L \) be a field extension and consider a \( K \)-subspace \( B \) of \( L \) with \( \dim B = n \). 
\begin{itemize}
\item $B$ is called a \emph{Chowla subspace} if 
\[
[K(x) : K] \geq n+1,
\]
for every nonzero \( x \in B \). 
\item Suppose \( d \) is an integer with \( 0 \leq d < n \). Then $B$ is called a \emph{$d$-defective Chowla subspace} if there exists a $K$-subspace \( B_0 \subseteq B \) with \( \dim B_0 \geq n - d \) such that 
\[
[K(x) : K] \geq n+1
\]
for every nonzero \( x \in B_0 \).  
\end{itemize}
\end{definition}

Note that a $d$-defective Chowla subspace with $d=0$ is simply a Chowla subspace.

The next corollary can be viewed as an extension of Corollary~3.6 of \cite{Aliabadi 5}. We remark that the notion of a Chowla set goes back to \cite{Hamidoune 2}, and that its relevance to matching theory was first observed in \cite{Hamidoune} and subsequently developed further in \cite{Aliabadi 5}.

\begin{corollary}
Let \( K \subsetneq L \) be a field extension and let \( 0 \leq d < n \) be integers. Suppose \( A \) and \( B \) are \( n \)-dimensional \( K \)-subspaces of \( L \) with \( 1 \notin B \). Suppose further that \( B \) is a $d$-defective Chowla subspace. Then \( A \) is partially matched to \( B \) with defect \( d \).
\end{corollary}

\begin{proof}
Since \( B \) is a $d$-defective Chowla subspace, there exists a subspace 
\( B_{0} \subseteq B \) with \( \dim B_{0} \geq n-d \) such that for every 
nonzero \( x \in B_{0} \),
\[
[K(x) : K] \geq n+1.
\]
Let \( S \subseteq A \) and \( R \subseteq B \oplus K \) be nonzero $K$-subspaces with \( SR = S\). We claim that \( R \cap B_{0} = \{ 0 \} \). Indeed, if this did not hold, we could choose a nonzero \( x \in R \cap B_{0} \), and then, by Proposition~\ref{proposition linear}, we would have  \( [K(x) : K] \leq n \) since \( x \in R \).  On the other hand, since \( x \in B_{0} \), we would have \( [K(x) : K] \geq n+1 \), a contradiction.  

Bearing in mind that $B\cap R$ and $B_{0}$ are contained in $B$ and that $(B\cap R)\cap B_{0}=\{0\}$, we have
\[
\dim(B\cap R)+\dim B_{0}\le n.
\]
It follows that
\begin{align*}
\dim S - d \leq n - d \leq \dim B_{0} 
\leq n - \dim (B \cap R)
=\dim \big(B/(B\cap R)\big).
\end{align*}
Applying Theorem~\ref{LinPartial}, we conclude that 
\( A \) is partially matched to \( B \) with defect \( d \).
\end{proof}

The following theorem turns the problem of computing \( \delta(A,B)\) into a simple maximization involving intersection dimensions. It also proves a corrected version of Conjecture 4.5 in \cite{Aliabadi 4} (see Remark \ref{ConjRem} below).

\begin{theorem}\label{DefTh}
Let $K\subsetneq L$ be a field extension, and let $A,B\subseteq L$ be $n$-dimensional $K$-subspaces with \( 0 < n < \infty \) and \( 1\notin B \). For a basis \( \mathcal{A}=\{a_1,\dots,a_n\}\) of \( A \) and for \( i\in [n]\),
define the $K$-subspace
\[
T_i(\mathcal{A}) \;=\; a_i^{-1}A\cap B \;\subseteq\; B.
\]
Define
\[
\Omega(\mathcal{A})
\;=\;
\max_{J\subseteq [n]}
\Bigl(
|J|+\dim \bigcap_{i\in J}T_i(\mathcal{A}) - n
\Bigr),
\]
where the intersection is understood to equal \( B \) when \( J = \emptyset \).
Then:
\begin{enumerate}
\item[\textup{(i)}] \( \Omega(\mathcal{A}) \) equals the least integer $d\ge 0$
for which there exists a basis $\mathcal{B}=\{b_1,\dots,b_n\}$ of $B$ such that
$\mathcal{A}$ is partially matched to $\mathcal{B}$ with defect $d$.

\item[\textup{(ii)}] Consequently,
\[
\delta(A,B)\;=\;\max_{\mathcal{A}} \Omega(\mathcal{A}),
\]
where \( \mathcal{A} \) runs over all ordered bases of \( A \).
\end{enumerate}
\end{theorem}

\begin{proof}
We first prove (i). Fix \( \mathcal{A}=\{a_1,\dots,a_n\} \) and write \( T_i=T_i(\mathcal{A})\). Let \( \mathcal{B}=\{b_1,\dots,b_n\} \) be a basis of \( B \) and let \( \{f_1,\dots,f_n\} \) be the corresponding dual basis of \(B^*\). For each $i$, the containment
\[
T_i \subseteq \langle b_1,\dots,b_{i-1},b_{i+1},\dots,b_n\rangle
\]
holds if and only if $f_i$ vanishes on $T_i$, i.e.,
\[
f_i\in T_i^\perp \subseteq B^*.
\]
Thus the existence of a basis \( \mathcal{B} \) for which \( \mathcal{A} \)  is partially matched to \( \mathcal{B} \) with defect \( d \) is equivalent to the existence of a basis \( \{f_1,\dots,f_n\} \)  of \( B^*\)  such that at least \( n-d \) of the vectors \( f_i \) belong to the corresponding subspaces \( T_i^\perp \).

Let \( U_i = T_i^\perp \).  Applying Theorem~\ref{Linear Partial} to the family \( (U_1,\dots,U_n) \) in \( B^* \), we find that the minimum number of indices that must be discarded in order for there to exist a linearly independent system \( f_i\in U_i \) from the remaining indices equals
\[
\max_{J\subseteq [n]}
\Bigl(|J|-\dim \sum_{i\in J}U_i
\Bigr).
\]
Using the annihilator identities given above Theorem~\ref{LinPartial}, we obtain
\begin{align*}
|J|-\dim \sum_{i\in J}U_i
=
|J|-\dim \sum_{i\in J}T_i^\perp 
&=
|J|-\dim\Bigl(\bigcap_{i\in J}T_i\Bigr)^\perp \\
&=
|J|-\Bigl(n-\dim \bigcap_{i\in J}T_i \Bigr) \\
&=
|J|+\dim \bigcap_{i\in J}T_i - n.
\end{align*}
Maximizing over $J$ yields $\Omega(\mathcal{A})$. Therefore, for the fixed basis \( \mathcal{A} \), the least integer \( d \geq 0 \) such that  \( (U_1,\dots,U_n) \) has a free partial transversal with defect \( d \) is \( \Omega(\mathcal{A}) \).  It now follows from the first paragraph above that \( \Omega(\mathcal{A}) \) is the least integer \( d \geq 0\) such that \( \mathcal{A} \) is partially matched to an ordered basis of \( B \) with defect \( d \). Thus \textup{(i)} is proved.

For \textup{(ii)}, let \( t = \delta(A,B) \). By definition, $t$ is the least integer $d$ such that for every ordered basis $\mathcal{A}$ of $A$ there exists an ordered basis of $B$ witnessing partial matchability with defect $d$. By \textup{(i)}, the least defect required for a given $\mathcal{A}$ equals $\Omega(\mathcal{A})$, hence $t$ must satisfy $t \ge \Omega(\mathcal{A})$ for all $\mathcal{A}$, so $t\ge \max_{\mathcal{A}}\Omega(\mathcal{A})$.
Conversely, if $t'=\max_{\mathcal{A}}\Omega(\mathcal{A})$, then $t' \ge \Omega(\mathcal{A})$ for each ordered basis $\mathcal{A}$, and so \textup{(i)} supplies, for each \( \mathcal{A} \), an ordered basis of $B$ witnessing partial matchability with defect $t'$ (see Remark~\ref{Defect Property}). Hence $\delta(A,B) = t \le t'$. Therefore $\delta(A,B)=\max_{\mathcal{A}}\Omega(\mathcal{A})$.
\end{proof}

\begin{remark} \label{ConjRem}
Let \(A,B,K,L\), and \(n\) be as in Theorem~\ref{DefTh}.  For any basis \( \mathcal{A}=\{a_1,\ldots,a_n\} \) of \(A\), Theorem~\ref{DefTh} shows that failure of matchability for \( \mathcal{A} \) is measured by the maximum Hall--Rado excess
\[
\Omega(\mathcal{A})
=
\max_{J\subseteq[n]}
\Big(
|J|
+
\dim\bigcap_{i\in J}(a_i^{-1}A\cap B)
-n
\Big),
\]
and that the minimum defect \( \delta(A,B) \) equals \( \max_{\mathcal{A}}\Omega(\mathcal{A}) \), where the maximum is taken over all ordered bases \( \mathcal{A} \) of \(A\).

Conjecture~4.5 in \cite{Aliabadi 4} instead considers, for a basis \( \mathcal{A}=\{a_1,\ldots,a_n\} \), the quantity
\[
D_{\mathcal{A}}(B)
=
\max\Big\{
|J|:
J\subseteq[n],\
\dim\bigcap_{i\in J}(a_i^{-1}A\cap B)>n-|J|
\Big\},
\]
and defines
\[
D(A,B)=\max_{\mathcal{A}}D_{\mathcal{A}}(B).
\]
Thus \(D_{\mathcal{A}}(B)\) records the cardinality of the largest index set violating the matching criterion, whereas \(\Omega(\mathcal{A})\) records the maximum amount by which the matching criterion fails. These two parameters need not be equal.  For example, let
\[
K=\mathbb{Q},\qquad
L=\mathbb{Q}(\sqrt{2},\sqrt{3}),\qquad
A=\mathbb{Q}(\sqrt{2}),\qquad
B=\langle\sqrt{2},\sqrt{3}\rangle_{\mathbb{Q}}.
\]
Here we have 
\[ n = \dim_{\mathbb{Q}} A = \dim_{\mathbb{Q}} B = 2, \qquad 1\notin B, \qquad  \langle AB\rangle \neq A.
\] 
(The last condition is a hypothesis of Conjecture~4.5 in \cite{Aliabadi 4}.)  For every basis \( \mathcal{A}=\{a_1,a_2\} \) of \(A\), each \(a_i\) is a nonzero element of the field \(A\), and hence \( a_i^{-1}A = A \).  Moreover, \( A\cap B=\langle\sqrt{2}\rangle_{\mathbb{Q}}\), so
\[
a_i^{-1}A\cap B
=
\langle\sqrt{2}\rangle_{\mathbb{Q}}
\qquad\text{for }i=1,2.
\]
It follows that
\[
D_{\mathcal{A}}(B)=2,
\qquad
\Omega(\mathcal{A})=1.
\]
Consequently,
\[
D(A,B)=2
\qquad\text{and}\qquad
\delta(A,B)=1.
\]
Thus the cardinality of the largest violating index set does not, in general, equal the minimum defect.

Following \cite{Aliabadi 4}, let \(M(A,B)\) denote the largest integer \(m\) for which \(A\) is partially matched to \(B\) with defect \(n-m\). In this example, 
\[
M(A,B)=n-\delta(A,B)=1,
\]
whereas the formula proposed in Conjecture~4.5 would give
\[
n-D(A,B)=2-2=0.
\]

In general, Theorem~\ref{DefTh} provides the corrected formula
\[
M(A,B) = n - \max_{\mathcal{A}}\Omega(\mathcal{A}),
\]
in which the maximum cardinality of a violating index set is replaced by the maximum Hall--Rado excess.
\end{remark}

\section{Structure theorems for partially matched subspaces}\label{LinStructure}

Below, we prove a decomposition theorem for pairs of subspaces with deficiency exceeding a prescribed threshold. This result may be regarded as a refinement of Theorem~3.3 in \cite{Aliabadi 6}.

\begin{theorem}\label{LinStrucTh}
Let \(K\subsetneq L\) be a field extension and let \(A,B\subseteq L\) be $n$-dimensional $K$-subspaces with \( 0 < n < \infty \) and \(1\notin B\). Let \(\ell\ge 0\) be an integer. Then \(\delta(A,B)>\ell\) if and only if there exists a nonzero $K$-subspace \(R\subseteq B\) such that \(A\) and \(B\) admit decompositions
\[
A = S \oplus Y,
\qquad
B = R \oplus Z,
\]
where $\dim Y<\dim R -\ell$ and
\[
S=\bigoplus_{i=1}^m a_iK(R)
\]
for some \( a_1, \ldots ,a_m \in S\). 
\end{theorem}

\noindent\emph{Note: }The inequality \( \dim Y<\dim R-\ell \) ensures that \( S\neq\{0\} \).  The value of \( m \) in the statement is \( \dim_{K(R)}S = (\dim_K S)/[K(R):K] \).

\begin{proof}
Assume that \(\delta(A,B)>\ell\). Then \(A\) is not partially matched to \(B\) with defect \(\ell\). From Theorem~\ref{LinPartial}, there exist nonzero $K$-subspaces \( S\subseteq A \) and \( R\subseteq B\oplus K \) such that
\( SR = S \) and
\begin{equation}\label{eq:thm41-ineq}
\dim S >\dim \bigl(B/(B\cap R)\bigr)+\ell.
\tag{3}
\end{equation}
Let \( R_0 = B\cap R \). Then \( R_0\neq \{0\} \), since otherwise \( \dim\big(B/(B\cap R)\big)=\dim B=\dim A \geq \dim S\), contradicting \eqref{eq:thm41-ineq}.

By Proposition~\ref{proposition linear}, for every \( a\in S \) we have \( aK(R)\subseteq S \). Since \( R_0\subseteq R \), it follows that \( K(R_0)\subseteq K(R) \), hence
\( aK(R_0)\subseteq S \) for all  \( a\in S \). Thus \( S \) can be viewed as a $K(R_0)$-vector space. Its $K(R_0)$-dimension is finite, since \( \dim_{K(R_0)}S \cdot [K(R_0):K] = \dim_K S \) and the latter is bounded above by \( n \).  Hence we can write 
\[
S=\bigoplus_{i=1}^m a_iK(R_0)
\]
for some \( a_1, \ldots ,a_m \in S\), with \( m = \dim_{K(R_0)}S = (\dim_K S)/[K(R_0):K] \). Choose $K$-subspaces \( Y \) of \( A \) and \( Z \) of \( B \) such that \( A=S\oplus Y \) and \( B=R_0\oplus Z \). Using $\dim \bigl(B/(B\cap R)\bigr)=\dim B-\dim R_0$, we rewrite inequality \eqref{eq:thm41-ineq} as
\[
\dim S>\dim B-\dim R_0+\ell.
\]
Since $\dim A=\dim B$, this implies
\[
\dim Y=\dim A-\dim S<\dim R_0-\ell,
\]
which is the desired conclusion.

Conversely, assume that $A$ and $B$ admit decompositions as in the statement. Define
\[
\widetilde{R}=R\oplus K \ \subseteq\ B\oplus K .
\]
Since $1\notin B$ and \( R \subseteq B \), it follows that $B\cap \widetilde{R}=R$. Also, we have $S\widetilde{R} = S$, on account of the fact that \( 1 \in \widetilde{R} \) and \( S \) is stable under multiplication by the elements of \( K(R) \).

The inequality $\dim Y < \dim R-\ell$ implies
\[
\dim S+\dim R-\ell > \dim S+\dim Y = \dim A = \dim B .
\]
Therefore 
\[
\dim S-\ell > \dim B-\dim R
= \dim(B/R) 
= \dim\bigl(B/(B\cap \widetilde{R})\bigr).
\]
By Theorem~3.3, $A$ is not partially matched to $B$ with defect $\ell$, and hence
$\delta(A,B)>\ell$.
This completes the proof.
\end{proof}

\begin{definition} \label{smallest finite nontrivial degree}
For a field extension \( K \subsetneq L \) with at least one nontrivial intermediate field of finite $K$-dimension, we define \( n_0(K, L) \) to be the smallest degree of such an extension, i.e.,
\[
n_0(K, L) = \min_F [F : K],
\]
where the minimum is taken over all intermediate fields \( K \subsetneq F \subseteq L \) such that \( [F : K] < \infty \). 
\end{definition}

In the theorem below, we give a necessary and sufficient condition for the existence of $K$-subspaces \(A,B\subseteq L\) of prescribed dimension with \(1\notin B\) and with deficiency \(\delta(A,B)\) exceeding a given threshold \(\ell\).

\begin{theorem}\label{ExistancePartial}
Let \( K \subsetneq L \) be a field extension with at least one proper nontrivial finite-dimensional intermediate field. Let \( \ell \) and \( n \) be integers such that \( \ell\geq 0 \) and \( n_0(K,L) \leq n < [L:K] \). Then there exist $n$-dimensional $K$-subspaces \( A,B \subseteq L \) with \( 1 \notin B \) and \( \delta(A,B)>\ell \) if and only if there exists an intermediate field \( K \subsetneq F \subsetneq L \) with \( [F:K] \leq n \) and \( [F:K] \nmid (n+j) \) for all \( j\in [\ell+1] \).
\end{theorem}

\begin{proof}
Assume that there exist $n$-dimensional $K$-subspaces \( A,B\subseteq L \) with \( 1\notin B \) such that \( \delta(A,B)>\ell \). By Theorem~\ref{LinStrucTh}, there exists a nonzero $K$-subspace $R\subseteq B$ such that $A$ and $B$ admit decompositions
\[
A = S \oplus Y,
\qquad
B = R \oplus Z,
\]
where  $\dim Y < \dim R - \ell$ and \( S \) is stable under multiplication by the elements of \( K(R) \).  Define
\[
\widetilde{R} = R + K \subseteq B \oplus K .
\]
Then \( S\widetilde{R} = S\), so Proposition~\ref{proposition linear} applies to the pair \((S,\widetilde{R})\). In particular,
\[
[K(R):K] = [K(\widetilde{R}):K] \leq n.
\]
We also have \( [K(R):K] > 1 \), since \( \{ 0 \} \neq R \subseteq B \) and \( 1 \notin B \).

Define $F=K(R)$ and write $[F:K]=m$. We just saw that 
\[ 1< m \leq n < [L:K], 
\] 
and now we wish to show that \( m \nmid (n+j) \) for all \( j \in [\ell + 1] \).  Reapplying Proposition \ref{proposition linear}, we obtain $\dim S = mq$ for some integer $q\ge 1$.
Set $r=\dim Y$. Then
\[
r=\dim Y=\dim A-\dim S=n-mq.
\]

Moreover, since $R\subseteq B$ and $1\notin B$, we have $1\notin R$, so $R$ is a
proper $K$-subspace of $F$, which implies $\dim R\le m-1$. Combining this with $\dim Y<\dim R-\ell$ yields
\[
r<\dim R-\ell\le (m-1)-\ell,
\]
and hence  $r\le m-\ell-2$. Therefore, for each \( j\in [\ell+1] \), we have
\[
n+j=mq+(r+j)
\]
with
\[
1\le r+j\le (m-\ell-2)+(\ell+1)=m-1.
\]
So the remainder of $n+j$ modulo $m$ is never $0$, i.e.,
\[
m\nmid(n+j)\qquad\text{for all }j\in[\ell+1].
\]
Thus \( F \) is the desired intermediate field.

Conversely, assume that there exists an intermediate field \( K\subsetneq F\subsetneq L \) with \( m=[F:K]\le n \) and \( m\nmid(n+j) \) for all \( j\in[\ell+1] \). Write \( n=mq+r \) using integers \( q, r\) with \( 0\le r<m \). The non-divisibility condition is equivalent to
\[
r+j\not\equiv 0 \pmod m \qquad \text{for all } j\in[\ell+1].
\]
Now, among the integers \( 1,2,\dots,m \) there is a unique value \( j_0 \) such that \( r+j_0\equiv 0\pmod m \), namely \( j_0=m-r \). Hence \( m-r\notin[\ell+1] \), i.e.,
\[
m-r>\ell+1,
\]
and therefore
\[
r\le m-\ell-2.
\]

Choose an $F$-subspace \( S\subseteq L \) of $F$-dimension \( q \); then \( \dim_K S=mq \). Next choose a $K$-subspace \( Y\subseteq L \) with \( \dim_K Y=r \) and \( S\cap Y=\{0\} \), and set \( A=S + Y \).  Then \( \dim_K A=mq+r=n \). Pick a $K$-hyperplane \( R\subseteq F \) with \( 1\notin R \), so \( \dim_K R=m-1 \).  Choose a $K$-subspace \( Z\subseteq L \) such that \( Z\cap F=\{0\} \) and \( \dim_K Z=n-(m-1)=n-m+1 \) (this can be done, since \( n < [L:K] \)), and set \( B=R + Z \). Then $\dim_K B=n$. We also have $1\notin B$ because if $1=x+z$ with $x\in R\subseteq F$ and $z\in Z$, then $z=1-x \in F$; hence $z\in Z\cap F=\{0\}$, forcing $x=1$, which contradicts $1\notin R$.

Since $K(R)\subseteq F$ and $S$ is an $F$-subspace, it follows that $S$ is a $K(R)$-subspace of $L$.  We can therefore write 
\[
S=\bigoplus_{i =1}^t a_iK(R)
\]
for some \( a_1,\ldots,a_t \in S \), with \( t = \dim_{K(R)} S = (\dim_K S)/[K(R):K] \).
Finally, using $r\le m-\ell-2$, we obtain
\[
\dim_K Y=r\le m-\ell-2<(m-1)-\ell=\dim_K R-\ell.
\]
Thus $A$ and $B$ admit decompositions $A=S\oplus Y$ and $B=R\oplus Z$ as in the statement of Theorem~\ref{LinStrucTh}.  We conclude that $\delta(A,B)>\ell$, completing the proof.
\end{proof}

Theorem~\ref{ExistancePartial} provides an existence criterion for pairs \((A,B)\) of $K$-dimension \(n\) with \(\delta(A,B)>\ell\), expressed via non-divisibility of the numbers \(n+1,\ldots,n+\ell+1\) by an intermediate degree \(m=[F:K] \leq n\). One may extract from this result a formula for the maximum achievable deficiency at dimension \(n\), expressed in terms of the set of intermediate degrees.  We present such a formula below, but we need a definition.

\begin{definition}
Let \( K \subsetneq L \) be a field extension with at least one proper nontrivial intermediate field \( F \) of finite $K$-dimension.  
\begin{itemize}
\item Define \( \mathcal{M}(K,L) \) to be the set of all degrees \( [F:K] \), where \( F \) is an intermediate field as above.  Note that \( n_0(K,L) \) from Definition~\ref{smallest finite nontrivial degree} equals \( \min \mathcal{M}(K,L) \) here.     
\item For each integer \( n \) satisfying \(n_0(K,L)\le n<[L:K]\), define
\[
\Delta(n)
=\max_{A,B} \delta(A,B),
\]
where \( A,B \) run over all pairs of $K$-subspaces of \( L \) with \( \dim A = \dim B = n \) and \( 1 \notin B \).
\end{itemize}
\end{definition}

\begin{theorem}\label{MaxDef}
Let $K\subsetneq L$ be a field extension with at least one proper nontrivial intermediate field
of finite $K$-dimension. For each integer \(n\) with \(n_0(K,L)\le n<[L:K]\), we have
\[
\Delta(n)
=
\max_{\substack{m\in \mathcal{M}(K,L)\\ m\le n}}
\Bigl(m-1-\bigl(n \bmod m\bigr)\Bigr).
\]
\end{theorem}

\begin{proof}
Fix an \( m\in \mathcal{M}(K,L) \) with \( m \leq n \).  Note that \( m \geq 2 \).  We show first that \( \Delta(n) \geq m - 1 - (n \bmod m) \). Write \( n=mq+r \) using integers \( q,r \) with \( 0\le r<m\).  If \( r=m-1 \), then \( m-r-1=0 \) and the inequality here is trivial since \( \Delta(n)\ge 0 \) by definition.  Assume henceforth that \( r\le m-2 \), and define \( \ell = m-r-2\ge 0 \).  Then the condition \( m\nmid (n+j) \) for all \( j\in[\ell +1] \) holds, so by Theorem~\ref{ExistancePartial} there exist $n$-dimensional $K$-subspaces \( A,B \) of \( L \) with \( 1 \notin B \) and \( \delta(A,B)>\ell \); hence \( \delta(A,B)\ge \ell +1=m-r-1 \). Thus, by the definition of \( \Delta(n) \), we have \( \Delta(n)\ge m-r-1 = m-1-(n \bmod m)\), as desired.

Maximizing over all admissible \(m\) gives us
\[
\Delta(n)\ge
\max_{\substack{m\in \mathcal{M}(K,L)\\ m\le n}}
\Bigl(m-1-\bigl(n\bmod m\bigr)\Bigr).
\]

For the reverse inequality, let \( t = \Delta(n) \) and let \( A,B \) be $n$-dimensional $K$-subspaces of \( L \) with \( 1 \notin B \) and \( \delta(A,B) = t \).  If \( t = 0 \), the desired inequality clearly holds, so assume \( t > 0 \). Then \(\delta(A,B)>t-1 \geq 0\), so by Theorem~\ref{ExistancePartial} there exists an intermediate field \(K\subsetneq F\subseteq L\) such that, setting \(m_0=[F:K] \), we have \( m_0 \le n\) and 
\[
m_0\nmid (n+j)\qquad\text{for all } j\in[t].
\]
Write \(n=m_0q'+r'\) using integers \( q',r' \) with \(0\le r'<m_0\). Since \(n+(m_0 -r')=m_0(q'+1)\) is divisible by \(m_0\), the above condition forces \(m_0-r'\notin[t]\), and hence
\[
t\le m_0-r'-1.
\]
We have shown that \(\Delta(n) = t \leq m_0-1-\bigl(n\bmod m_0\bigr) \) for a specific \( m_0 \in \mathcal{M}(K,L) \) with \( m_0 \leq n \).  It follows that  
\[
\Delta(n)\le
\max_{\substack{m\in \mathcal{M}(K,L)\\ m\le n}}
\Bigl(m-1-\bigl(n\bmod m\bigr)\Bigr),
\]
completing the proof.
\end{proof}

We conclude with a result concerning finite fields.

\begin{corollary}
Consider the field extension \(\mathbb{F}_q\subseteq \mathbb{F}_{q^N}\), where \( q \) is a prime power and \(N>1\) is a composite number.  Let \(p\) denote the smallest prime divisor of \(N\).
Then for each integer \( n \) satisfying \(p\le n< N\), we have
\[
\Delta(n)
=
\max_{\substack{m\mid N\\ 2\le m\le n}}
\Bigl(m-1-\bigl(n\bmod m\bigr)\Bigr).
\]
\end{corollary}

\begin{proof}
This follows easily from Theorem~\ref{MaxDef}, since the nontrivial intermediate fields of the extension \(\mathbb{F}_q \subseteq \mathbb{F}_{q^{N}}\) are exactly the fields of the form \( \mathbb{F}_{q^{m}} \), where \( m \geq 2 \) is a divisor of \( N \), and since the degree of \( \mathbb{F}_{q^m} \) over \( \mathbb{F}_q \) is \( m \).
\end{proof}

\medskip

\begin{remark}
The theory developed in this paper on partial matchability relies on commutativity at several junctures, most notably in the construction of a linearized Dyson \(e\)-transform. Nevertheless, we expect that many of the underlying ideas admit a meaningful extension to noncommutative settings. A natural starting point would be to replace the ambient field \( L \) by a division ring \( D \) and assume that \( K \) is a subfield of the center \( Z(D) \) of \( D \). Developing a parallel theory in this context, and in particular establishing genuinely noncommutative analogues of our main results, seems a promising direction for future work.
\end{remark}

\medskip

\noindent \textbf{Declarations} 

\medskip

\noindent \textbf{\small Conflict of interest:} {\small The authors declare that they have no conflict of interest.}
\noindent \textbf{\small Data availability statement:} {\small Data sharing is not applicable to this article as no datasets were generated or analyzed during the current study.}

\enlargethispage{\baselineskip}


\begin{thebibliography}{10}

\bibitem{Aliabadi 0}
M. Aliabadi and K. Filom, Results and questions on matchings in abelian groups and vector subspaces of fields, \textit{J. Algebra} 598 (2022) 85--104.

\bibitem{Aliabadi 4}
M. Aliabadi, J. Kinseth, C. Kunz, H. Serdarevic and C. Willis, Conditions for matchability in groups and field extensions, \emph{Linear Multilinear Algebra} 71 (7) (2023) 1182--1197.

\bibitem{Aliabadi 5}
M. Aliabadi and J. Losonczy, Characterization of matchable sets and subspaces via Dyson transforms, \emph{J.\ Algebra} 706 (2026) 221--242.

\bibitem{Aliabadi 6}
M. Aliabadi and J. Losonczy, Decomposition theorems for unmatchable pairs in groups and field extensions, to appear in \textit{Ann. Combin.},
https://arxiv.org/abs/2512.12942.

\bibitem{Aliabadi 7}
M. Aliabadi and J. Losonczy, Structure and decomposition of deltoids in abelian groups, https://arxiv.org/abs/2601.09774.

\bibitem{Taylor}
M. Aliabadi and P. Taylor, Classifying abelian groups through acyclic matchings, \textit{Ann. Combin.} 29 (2025) 1019--1025.

\bibitem{Aliabadi 3}
M. Aliabadi and S. Zerbib, Matchings in matroids over abelian groups, \textit{J. Algebraic Combin.} 59 (2024) 761--785.

\bibitem{Alon} 
N. Alon, C. K. Fan, D. Kleitman and J. Losonczy, Acyclic matchings, \textit{Adv. Math.} 122 (2) (1996) 234--236.

\bibitem{Ehrenborg}
R. Ehrenborg, On apolarity and generic canonical forms, \textit{J. Algebra} 213 (1) (1999) 167--194.

\bibitem{Ehrenborg Rota}
R. Ehrenborg and G.-C. Rota, Apolarity and canonical forms for homogeneous polynomials, \textit{European J. Combin.} 14 (1993) 157--191. 

\bibitem{Eliahou 1}
S. Eliahou and C. Lecouvey, Matchings in arbitrary groups, \textit{Adv. in Appl. Math.} 40 (2008) 219--224.
 
\bibitem{Eliahou 2}
S. Eliahou and C. Lecouvey, Matching subspaces in a field extension, \textit{J. Algebra} 324 (2010) 3420--3430.

\bibitem{Losonczy 1}
C. K. Fan and J. Losonczy, Matchings and canonical forms for symmetric tensors, \textit {Adv. Math.} 117 (2) (1996) 228--238.

\bibitem{Hamidoune 2} 
Y. O. Hamidoune, An isoperimetric method in additive theory, \textit{J. Algebra} 179 (2) (1996) 622--630.
 
\bibitem{Hamidoune}
Y. O. Hamidoune, Counting certain pairings in arbitrary groups, \textit{Combin. Probab. Comput.} 20 (6) (2011) 855--865.

\bibitem{Hamidoune 3}
Y. O. Hamidoune, Extensions of the Scherk-Kemperman theorem, \textit{J. Combin. Theory Ser. A} 117 (7) (2010) 974--980.
 
\bibitem{Levi}
F. W. Levi, Ordered groups, \textit{Proc. Indian Acad. Sci.} 16 (1942) 256--263.

\bibitem{Losonczy 2}
J. Losonczy, On matchings in groups, \textit{Adv. in Appl. Math.} 20 (3) (1998) 385--391.

\bibitem{Mirsky}
L.\ Mirsky, \textit{Transversal Theory}, Academic Press, New York, 1971. 
  
\bibitem{Wakeford}
E. K. Wakeford, On canonical forms, \textit{Proc. Lond. Math. Soc.} (2) 18 (1920) 403--410.

\end{thebibliography}
\end{document}